\documentclass[11pt]{amsart}
\usepackage{fullpage}
\usepackage[active]{srcltx}
\usepackage{amssymb,amsthm,amsmath,amscd}

\usepackage{graphicx} % Required for inserting images
\usepackage{amsfonts}
\usepackage{enumitem,kantlipsum}
\usepackage{amsmath}
\usepackage{amssymb}
\usepackage{amsthm}
\newtheorem{theorem}{Theorem}

\newtheorem{claim}[theorem]{Claim}

\theoremstyle{definition}
\newtheorem{remark}[theorem]{Remark}

\title{Infinitely many associated primes of local cohomology modules of ramified regular local rings}
\author{Linquan Ma}
\address{Department of Mathematics, Purdue University, West Lafayette, IN 47907, USA}
\email{ma326@purdue.edu}
\thanks{The author was supported by NSF grants DMS-2302430 and DMS-2424441. I would like to thank Bhargav Bhatt, Anurag Singh, and Uli Walther for many enjoyable discussions on local cohomology.}
\begin{document}

\begin{abstract}
We construct examples of local cohomology modules of ramified regular local rings with infinitely many associated primes and infinite Bass numbers.
\end{abstract}

\maketitle

A question of Lyubeznik asks whether the local cohomology modules of regular rings have only finitely many associated primes and finite Bass numbers \cite[Remark 3.7 and Question 3.8]{LyubeznikDmodules}. There have been extensive works on this question. In particular, the question has an affirmative answer for regular local rings containing a field or of unramified mixed characteristic \cite{HunekeSharp,LyubeznikDmodules,LyubeznikFinitentessUnramifiedCase,NunezBetancourtLocalCohomologySmallDimension}. See  \cite{NunezBetancourtRingsDifferentialType,BBLSZ1,CidRuizSmirnov,MurayamaAssociatedPrimesExcellent} for some other major progress.

Here we give a negative answer to Lyubeznik's question for (ramified) regular local rings of mixed characteristic, which also disproves \cite[Conjecture 5.2]{HunekeProblemsLocalCohomology}. We combine the examples studied in \cite{DattaSwitalaZhang} (see also \cite{HNBPWInjectiveDimension}) with the examples in \cite{KatzmanInfiniteAssociatedPrimes,SinghSwansonAssociatedPrimes}.

\subsection*{Construction} Let $\mathbb{Z}_2$ be the ring of $2$-adic integers and let $I\subseteq \mathbb{Z}_2[[\underline{x}]] :=\mathbb{Z}_2[[x_1,\dots, x_6]]$ be the monomial ideal corresponding to the standard triangulation of $\mathbb{R}\mathbb{P}^2$ as in \cite[Example 5.2]{SinghWaltherBockstein}:
\[
  I=( x_1x_2x_3,\;
    x_1x_2x_4,\;
    x_1x_3x_5,\;
    x_1x_4x_6,\;
    x_1x_5x_6,\;
    x_2x_3x_6,\;
    x_2x_4x_5,\;
    x_2x_5x_6,\;
    x_3x_4x_5,\;
    x_3x_4x_6
  ).
\]
It follows from \cite[Proposition 4.5 and Remark 4.3]{DattaSwitalaZhang} that 
$$H_I^4(\mathbb{Z}_2[[\underline{x}]])\cong \mathrm{Ann}_E(2) \cong E_{\mathbb{F}_2[[\underline{x}]]}(\mathbb{F}_2) \,\ \text{ and } \,\ H_I^5(\mathbb{Z}_2[[\underline{x}]])=0$$
where $E=E_{\mathbb{Z}_2[[\underline{x}]]}(\mathbb{F}_2)$ denotes the injective hull of the residue field of $\mathbb{Z}_2[[\underline{x}]]$.

Next, we recall the example of a local cohomology module of a hypersurface with infinitely many associated primes from \cite{SinghSwansonAssociatedPrimes} (we can also use the example in \cite{KatzmanInfiniteAssociatedPrimes}). By \cite[Theorem 4.1]{SinghSwansonAssociatedPrimes} (followed by a harmless completion map), we have that 
\begin{equation}
\label{eqn: SinghSwansonExample}
H_{(y_1, y_2)}^2\left(\frac{\mathbb{F}_2[[y_1,\dots,y_6]]}{(y_1^2y_3^2y_6 + y_1y_2y_3y_4y_5 + y_2^2y_4^2y_6)}\right) \tag{$\dagger$}
\end{equation}
has infinitely many associated primes.

%let $g(\underline{y}):=g(y_1,\dots,y_6)\subseteq \mathbb{F}_2[y_1,\dots,y_6]=:\mathbb{F}_2[\underline{y}]$ be Katzman's example so that $H_{(y_1,y_2)}^2(\mathbb{F}_2[\underline{y}]/g(\underline{y}))$ has infinitely many associated primes. We lift the coefficients of $g(y)$ from $\mathbb{F}_2$ to $\mathbb{Z}_2$ in an arbitrary way, and we denote that lift $f(\underline{y})\in \mathbb{Z}_2[\underline{y}]$.

Now we set $S=\mathbb{Z}_2[[\underline{x},\underline{y}]]:=\mathbb{Z}_2[[x_1,\dots,x_6, y_1,\dots,y_6]]$ and let
$$R:=S/(2+y_1^2y_3^2y_6 + y_1y_2y_3y_4y_5 + y_2^2y_4^2y_6).$$ 
Note that $R$ is a complete ramified regular local ring. Finally, we let $J :=IS + (y_1, y_2)S\subseteq S.$

\begin{claim}
$H_{J}^6(R)$ has infinitely many associated primes.
\end{claim}
\begin{proof}
First of all, it is well-known and easy to show that for any ring $A$ and any ideal $\mathfrak{a}\subseteq A$, we have $H^{i+2}_{\mathfrak{a}+(z_1, z_2)}(A[[z_1, z_2]])\cong H_{\mathfrak{a}}^i(A) \otimes_{\mathbb{Z}}H_{(z_1, z_2)}^2(\mathbb{Z}[[z_1,z_2]])$ (for example, we can apply \cite[Discussion 4.12]{EnescuHochster} twice). Applying this followed by the flat base change of adjoining more indeterminates, we see that by  construction ($\mathbb{Z}_2[[\underline{y}]]= \mathbb{Z}_2[[y_1,\dots,y_6]]$ and $\mathbb{F}_2[[\underline{y}]]= \mathbb{F}_2[[y_1,\dots,y_6]]$ in what follows):
$$H_J^{6}(S) \cong H_{I}^4(\mathbb{Z}_2[[\underline{x}]])\otimes_{\mathbb{Z}}H_{(y_1,y_2)}^2(\mathbb{Z}[[\underline{y}]])\cong E_{\mathbb{F}_2[[\underline{x}]]}(\mathbb{F}_2) \otimes_{\mathbb{F}_2}H_{(y_1,y_2)}^2(\mathbb{F}_2[[\underline{y}]])$$
as $S$-modules, and
$$H^{7}_J(S)\cong H_{I}^5(\mathbb{Z}_2[[\underline{x}]])\otimes_{\mathbb{Z}}H_{(y_1,y_2)}^2(\mathbb{Z}[[\underline{y}]])=0.$$
Denote $f(\underline{y}):= y_1^2y_3^2y_6 + y_1y_2y_3y_4y_5 + y_2^2y_4^2y_6$ and consider part of the long exact sequence
$$H_J^6(S) \xrightarrow{\cdot (2+f(\underline{y}))} H_J^6(S) \to H_J^6(R) \to H_J^7(S)=0.$$
It follows that $H_J^6(R)$ is the cokernel of the multiplication by $(2+f(\underline{y}))$ map on $H_J^6(S)$. But since $H_J^{6}(S)\cong E_{\mathbb{F}_2[[\underline{x}]]}(\mathbb{F}_2) \otimes_{\mathbb{F}_2}H_{(y_1,y_2)}^2(\mathbb{F}_2[[\underline{y}]])$ and in particular $H_J^6(S)$ is annihilated by $(2)$, multiplication by $(2+f(\underline{y}))$ is the same as multiplication by $f(\underline{y})$ and we have
\begin{align*}
H_J^6(R) \cong & E_{\mathbb{F}_2[[\underline{x}]]}(\mathbb{F}_2) \otimes_{\mathbb{F}_2}\mathrm{coker}\big(H_{(y_1,y_2)}^2(\mathbb{F}_2[[\underline{y}]])\xrightarrow{\cdot f(\underline{y})}H_{(y_1,y_2)}^2(\mathbb{F}_2[[\underline{y}]])\big) \\
\cong & E_{\mathbb{F}_2[[\underline{x}]]}(\mathbb{F}_2) \otimes_{\mathbb{F}_2} H_{(y_1,y_2)}^2\big(\mathbb{F}_2[[\underline{y}]]/f(\underline{y})\big)
\end{align*}
as $S$-modules (and thus also as $R$-modules), where we used the fact that the cohomological dimension of $(y_1, y_2)$ is two for the second isomorphism.
%Finally, since the image of $f(\underline{y})$ in $\mathbb{F}_2[\underline{y}]$ is $g(\underline{y})$ and we obviously have 
%$$\mathrm{coker}(H_{(y_1,y_2)}^2(\mathbb{F}_2[\underline{y}])\xrightarrow{\cdot g(\underline{y})}H_{(y_1,y_2)}^2(\mathbb{F}_2[\underline{y}]))\cong H_{(y_1,y_2)}^2(\mathbb{F}_2[\underline{y}]/g(\underline{y})).$$
%It follows that 
%$$H_J^6(R)\cong E_{\mathbb{F}_2[\underline{x}]}(\mathbb{F}_2) \otimes_{\mathbb{F}_2} H_{(y_1,y_2)}^2(\mathbb{F}_2[\underline{y}]/g(\underline{y})).$$

Since $H_J^6(R)$ is annihilated by $(2)$, it is enough to show that $H_J^6(R)$ has infinitely many associated primes as a module over $R/(2)$. But we have 
$$R/(2) \cong \frac{\mathbb{F}_2[[y_1,\dots,y_6]]}{(y_1^2y_3^2y_6 + y_1y_2y_3y_4y_5 + y_2^2y_4^2y_6)}[[x_1,\dots, x_6]].$$
By the explicit description of $H_J^6(R)$ above, it is clear that for each of the infinitely many associated prime $Q$ of the local cohomology module in (\ref{eqn: SinghSwansonExample}), 
$Q+(x_1,\dots,x_6)$ is an associated prime of $H_J^6(R)$ over $R/(2)$.  In particular, $H_J^6(R)$ has infinitely many associated primes.
%and it is clear by construction that the action of $R/(2)$ on 
%$$H_J^6(R) \cong E_{\mathbb{F}_2[[\underline{x}]]}(\mathbb{F}_2) \otimes_{\mathbb{F}_2} H_{(y_1, y_2)}^2\left(\frac{\mathbb{F}_2[[y_1,\dots,y_6]]}{(y_1^2y_3^2y_6 + y_1y_2y_3y_4y_5 + y_2^2y_4^2y_6)}\right)$$
%It is now clear that for any associated prime $Q$ of $H_{(y_1,y_2)}^2(\mathbb{F}_2[\underline{y}]/g(\underline{y}))$ in the ring $\mathbb{F}_2[\underline{y}]$, $Q$ corresponds uniquely to a prime $\widetilde{Q}\subseteq \mathbb{Z}_2[\underline{y}]$ containing $(2)$ and $\widetilde{Q}S+(\underline{x})S$ is an associated prime of $H_J^6(S)$. In particular, $H_J^6(S)$ has infinitely many associated primes.
\end{proof}

The same construction also leads to examples of local cohomology modules of ramified regular local rings with infinite dimensional socles (and in particular infinite Bass numbers),\footnote{Wenliang Zhang kindly informed us that he has independently obtained the same example in Claim 2.} which disproves \cite[Conjecture 4.3 and 4.4]{HunekeProblemsLocalCohomology}. Let $T=\mathbb{Z}_2[[\underline{x}, \underline{y}]] := \mathbb{Z}_2[[x_1,\dots,x_6, y_1,\dots,y_4]]$ and $$R:=T/(2+y_1y_3+y_2y_4) \,\  \,\ J:=IT+(y_1,y_2)T\subseteq T.$$

\begin{claim}
$H_{J}^6(R)$ has infinite dimensional socle.
\end{claim}
\begin{proof}
By the same argument as above we have
$$H_J^6(R)\cong E_{\mathbb{F}_2[[\underline{x}]]}(\mathbb{F}_2) \otimes_{\mathbb{F}_2} H_{(y_1,y_2)}^2(\mathbb{F}_2[[\underline{y}]]/(y_1y_3+y_2y_4)).$$
By \cite[Section 3]{HartshorneAffineDualityCofiniteness} (and note that the result there still holds after a harmless completion), we know that $H_{(y_1,y_2)}^2(\mathbb{F}_2[[\underline{y}]]/(y_1y_3+y_2y_4))$ has infinite dimensional socle as an $\mathbb{F}_2[[\underline{y}]]$-module. It follows that $H_J^6(R)$ has infinite dimensional socle as an $R$-module.
\end{proof}

\begin{remark} We point out the following:
\begin{enumerate}[leftmargin=*]
\item[(a)] We can also obtain examples that are finite type over a DVR or over $\mathbb{Z}$, simply take $$R=W[x_1,\dots,x_6, y_1, \dots, y_6]/(2+y_1^2y_3^2y_6 + y_1y_2y_3y_4y_5 + y_2^2y_4^2y_6)$$
    where $W=\mathbb{Z}$, $\mathbb{Z}_{(2)}$, or the ring of Witt vectors of any perfect field of characteristic $2$. In fact, in all these cases we can localize $R$ at $(2, \underline{x},\underline{y})$ and the proof above will not be altered much to show that the associated primes of $H_J^6(R)$ that are contained in this maximal ideal are infinite.
\item[(b)] Similar examples exist for every prime residue characteristic $p$: we simply  replace the monomial ideal corresponding to the triangulation of $\mathbb{R}\mathbb{P}^2$ by the monomial ideal corresponding to the triangulation of the $p$-fold dunce cap as in \cite[Example 5.11]{SinghWaltherBockstein} (the argument in \cite[Proposition 4.5]{DattaSwitalaZhang} goes through and shows that the critical local cohomology is annihilated by $p$, and is in fact isomorphic to $\mathrm{Ann}_E(p)$). We leave the details to the readers.
\end{enumerate}
\end{remark}

\bibliographystyle{alpha}
\bibliography{Bib}

@article {DattaSwitalaZhang,
    AUTHOR = {Datta, Rankeya and Switala, Nicholas and Zhang, Wenliang},
     TITLE = {Annihilators of {$\mathcal{D}$}-modules in mixed characteristic},
   JOURNAL = {Math. Res. Lett.},
  FJOURNAL = {Mathematical Research Letters},
    VOLUME = {30},
      YEAR = {2023},
    NUMBER = {3},
     PAGES = {721--732},
      ISSN = {1073-2780,1945-001X},
   MRCLASS = {14F10 (13A35 13D45 14G45)},
  MRNUMBER = {4696428},
MRREVIEWER = {Sebasti\'{a}n\ Olano},
       DOI = {10.4310/mrl.2023.v30.n3.a5},
       URL = {https://doi.org/10.4310/mrl.2023.v30.n3.a5},
}

@article {SinghSwansonAssociatedPrimes,
    AUTHOR = {Singh, Anurag K. and Swanson, Irena},
     TITLE = {Associated primes of local cohomology modules and of
              {F}robenius powers},
   JOURNAL = {Int. Math. Res. Not.},
  FJOURNAL = {International Mathematics Research Notices},
      YEAR = {2004},
    NUMBER = {33},
     PAGES = {1703--1733},
      ISSN = {1073-7928,1687-0247},
   MRCLASS = {13D45},
  MRNUMBER = {2058025},
MRREVIEWER = {Manuel\ Blickle},
       DOI = {10.1155/S1073792804133424},
       URL = {https://doi.org/10.1155/S1073792804133424},
}

@article {HartshorneAffineDualityCofiniteness,
    AUTHOR = {Hartshorne, Robin},
     TITLE = {Affine duality and cofiniteness},
   JOURNAL = {Invent. Math.},
  FJOURNAL = {Inventiones Mathematicae},
    VOLUME = {9},
      YEAR = {1969/70},
     PAGES = {145--164},
      ISSN = {0020-9910,1432-1297},
   MRCLASS = {14.55},
  MRNUMBER = {257096},
MRREVIEWER = {D.\ S.\ Rim},
       DOI = {10.1007/BF01404554},
       URL = {https://doi.org/10.1007/BF01404554},
}

@article {KatzmanInfiniteAssociatedPrimes,
    AUTHOR = {Katzman, Mordechai},
     TITLE = {An example of an infinite set of associated primes of a local
              cohomology module},
   JOURNAL = {J. Algebra},
  FJOURNAL = {Journal of Algebra},
    VOLUME = {252},
      YEAR = {2002},
    NUMBER = {1},
     PAGES = {161--166},
      ISSN = {0021-8693,1090-266X},
   MRCLASS = {13D45},
  MRNUMBER = {1922391},
MRREVIEWER = {I-Chiau\ Huang},
       DOI = {10.1016/S0021-8693(02)00032-7},
       URL = {https://doi.org/10.1016/S0021-8693(02)00032-7},
}

@article {SinghWaltherBockstein,
    AUTHOR = {Singh, Anurag K. and Walther, Uli},
     TITLE = {Bockstein homomorphisms in local cohomology},
   JOURNAL = {J. Reine Angew. Math.},
  FJOURNAL = {Journal f\"{u}r die Reine und Angewandte Mathematik. [Crelle's
              Journal]},
    VOLUME = {655},
      YEAR = {2011},
     PAGES = {147--164},
      ISSN = {0075-4102,1435-5345},
   MRCLASS = {13D45 (13F20 13F55 13H05)},
  MRNUMBER = {2806109},
MRREVIEWER = {D.-M.\ Popescu},
       DOI = {10.1515/CRELLE.2011.039},
       URL = {https://doi.org/10.1515/CRELLE.2011.039},
}

@article {LyubeznikDmodules,
    AUTHOR = {Lyubeznik, Gennady},
     TITLE = {Finiteness properties of local cohomology modules (an
              application of {$D$}-modules to commutative algebra)},
   JOURNAL = {Invent. Math.},
  FJOURNAL = {Inventiones Mathematicae},
    VOLUME = {113},
      YEAR = {1993},
    NUMBER = {1},
     PAGES = {41--55},
      ISSN = {0020-9910,1432-1297},
   MRCLASS = {13D45 (13H05 14B12)},
  MRNUMBER = {1223223},
MRREVIEWER = {P.\ Schenzel},
       DOI = {10.1007/BF01244301},
       URL = {https://doi.org/10.1007/BF01244301},
}

@article {HunekeSharp,
    AUTHOR = {Huneke, Craig L. and Sharp, Rodney Y.},
     TITLE = {Bass numbers of local cohomology modules},
   JOURNAL = {Trans. Amer. Math. Soc.},
  FJOURNAL = {Transactions of the American Mathematical Society},
    VOLUME = {339},
      YEAR = {1993},
    NUMBER = {2},
     PAGES = {765--779},
      ISSN = {0002-9947,1088-6850},
   MRCLASS = {13D45},
  MRNUMBER = {1124167},
MRREVIEWER = {Gennady\ Lyubeznik},
       DOI = {10.2307/2154297},
       URL = {https://doi.org/10.2307/2154297},
}

@article{LyubeznikFinitentessUnramifiedCase,
    AUTHOR = {Lyubeznik, Gennady},
     TITLE = {Finiteness properties of local cohomology modules for regular
              local rings of mixed characteristic: the unramified case},
     NOTE = {Special issue in honor of Robin Hartshorne},
   JOURNAL = {Comm. Algebra},
  FJOURNAL = {Communications in Algebra},
    VOLUME = {28},
      YEAR = {2000},
    NUMBER = {12},
     PAGES = {5867--5882},
      ISSN = {0092-7872,1532-4125},
   MRCLASS = {13D45},
  MRNUMBER = {1808608},
MRREVIEWER = {David\ A.\ Jorgensen},
       DOI = {10.1080/00927870008827193},
       URL = {https://doi.org/10.1080/00927870008827193},
}

@article {HNBPWInjectiveDimension,
    AUTHOR = {Hern\'{a}ndez, Daniel J. and N\'{u}\~{n}ez-Betancourt, Luis
              and P\'{e}rez, Felipe and Witt, Emily E.},
     TITLE = {Lyubeznik numbers and injective dimension in mixed
              characteristic},
   JOURNAL = {Trans. Amer. Math. Soc.},
  FJOURNAL = {Transactions of the American Mathematical Society},
    VOLUME = {371},
      YEAR = {2019},
    NUMBER = {11},
     PAGES = {7533--7557},
      ISSN = {0002-9947,1088-6850},
   MRCLASS = {13A35 (13C11 13D45)},
  MRNUMBER = {3955527},
MRREVIEWER = {Geoffrey\ D.\ Dietz},
       DOI = {10.1090/tran/7310},
       URL = {https://doi.org/10.1090/tran/7310},
}

@incollection {HunekeProblemsLocalCohomology,
    AUTHOR = {Huneke, Craig},
     TITLE = {Problems on local cohomology},
 BOOKTITLE = {Free resolutions in commutative algebra and algebraic geometry
              ({S}undance, {UT}, 1990)},
    SERIES = {Res. Notes Math.},
    VOLUME = {2},
     PAGES = {93--108},
 PUBLISHER = {Jones and Bartlett, Boston, MA},
      YEAR = {1992},
      ISBN = {0-86720-285-8},
   MRCLASS = {13D45 (13H10 14B15)},
  MRNUMBER = {1165320},
MRREVIEWER = {D.-M.\ Popescu},
}

@article {NunezBetancourtLocalCohomologySmallDimension,
    AUTHOR = {N\'{u}\~{n}ez-Betancourt, Luis},
     TITLE = {Local cohomology modules of polynomial or power series rings
              over rings of small dimension},
   JOURNAL = {Illinois J. Math.},
  FJOURNAL = {Illinois Journal of Mathematics},
    VOLUME = {57},
      YEAR = {2013},
    NUMBER = {1},
     PAGES = {279--294},
      ISSN = {0019-2082,1945-6581},
   MRCLASS = {13D45},
  MRNUMBER = {3224571},
MRREVIEWER = {Majid\ Eghbali},
       URL = {http://projecteuclid.org/euclid.ijm/1403534496},
}

@article {NunezBetancourtRingsDifferentialType,
    AUTHOR = {N\'{u}\~{n}ez-Betancourt, Luis},
     TITLE = {On certain rings of differentiable type and finiteness
              properties of local cohomology},
   JOURNAL = {J. Algebra},
  FJOURNAL = {Journal of Algebra},
    VOLUME = {379},
      YEAR = {2013},
     PAGES = {1--10},
      ISSN = {0021-8693,1090-266X},
   MRCLASS = {13D45 (13N10)},
  MRNUMBER = {3019242},
MRREVIEWER = {Alberto\ F.\ Boix},
       DOI = {10.1016/j.jalgebra.2012.12.010},
       URL = {https://doi.org/10.1016/j.jalgebra.2012.12.010},
}

@article {BBLSZ1,
    AUTHOR = {Bhatt, Bhargav and Blickle, Manuel and Lyubeznik, Gennady and
              Singh, Anurag K. and Zhang, Wenliang},
     TITLE = {Local cohomology modules of a smooth {$\Bbb{Z}$}-algebra have
              finitely many associated primes},
   JOURNAL = {Invent. Math.},
  FJOURNAL = {Inventiones Mathematicae},
    VOLUME = {197},
      YEAR = {2014},
    NUMBER = {3},
     PAGES = {509--519},
      ISSN = {0020-9910,1432-1297},
   MRCLASS = {13D45},
  MRNUMBER = {3251828},
MRREVIEWER = {Le\ Xuan\ Dung},
       DOI = {10.1007/s00222-013-0490-z},
       URL = {https://doi.org/10.1007/s00222-013-0490-z},
}

@ARTICLE{MurayamaAssociatedPrimesExcellent,
       author = {{Murayama}, Takumi},
        title = "{Finiteness of associated primes for local cohomology modules of excellent locally unramified regular rings of finite Krull dimension}",
      journal = {arXiv e-prints},
         year = 2025,
        month = jun,
          eid = {arXiv:2506.17875},
        note = {arXiv:2506.17875},
          doi = {10.48550/arXiv.2506.17875},
archivePrefix = {arXiv},
       eprint = {2506.17875},
 primaryClass = {math.AC},
       adsurl = {https://ui.adsabs.harvard.edu/abs/2025arXiv250617875M}
}

@article {EnescuHochster,
    AUTHOR = {Enescu, Florian and Hochster, Melvin},
     TITLE = {The {F}robenius structure of local cohomology},
   JOURNAL = {Algebra Number Theory},
  FJOURNAL = {Algebra \& Number Theory},
    VOLUME = {2},
      YEAR = {2008},
    NUMBER = {7},
     PAGES = {721--754},
      ISSN = {1937-0652,1944-7833},
   MRCLASS = {13A35 (13D45 13F55)},
  MRNUMBER = {2460693},
MRREVIEWER = {Irena\ Swanson},
       DOI = {10.2140/ant.2008.2.721},
       URL = {https://doi.org/10.2140/ant.2008.2.721},
}

@article {CidRuizSmirnov,
    AUTHOR = {Cid-Ruiz, Yairon and Smirnov, Ilya},
     TITLE = {Effective generic freeness and applications to local
              cohomology},
   JOURNAL = {J. Lond. Math. Soc. (2)},
  FJOURNAL = {Journal of the London Mathematical Society. Second Series},
    VOLUME = {110},
      YEAR = {2024},
    NUMBER = {4},
     PAGES = {Paper No. e12995, 31},
      ISSN = {0024-6107,1469-7750},
   MRCLASS = {13C10 (13B40 13D45 13P10)},
  MRNUMBER = {4801897},
MRREVIEWER = {Irena\ Swanson},
       DOI = {10.1112/jlms.12995},
       URL = {https://doi.org/10.1112/jlms.12995},
}
\end{document}